\documentclass[12pt]{article}
\usepackage{amsmath, amssymb, amsthm, graphicx, tikz}
\usepackage[margin=1in]{geometry}
\usepackage{enumitem}
\usepackage{forest}
\usepackage{tikz}
\usetikzlibrary{trees}
\usepackage{setspace}
\usepackage{xcolor}
\usepackage{tcolorbox}
\usepackage{float}
\usetikzlibrary{backgrounds}   
\usetikzlibrary{positioning}    
\usetikzlibrary{fit}            
\usepackage[utf8]{inputenc}
\usepackage{graphicx}   
\usepackage{caption}    
\usepackage{subcaption} 
 \usepackage{geometry} 
\usepackage{amsthm}

\usepackage[T1]{fontenc}
\usepackage{subcaption}

\usepackage{tikz}
\usetikzlibrary{shapes.geometric, positioning, calc}

\usetikzlibrary{calc}
 \definecolor{myblue}{RGB}{0,0,255}
 \usepackage{pstricks}
\usepackage{pst-node}
\usepackage{pst-plot}
 
\numberwithin{equation}{section}
\numberwithin{figure}{section}
\numberwithin{table}{section}

\newtheorem{theorem}{Theorem}[section]

\newtheorem{proposition}[theorem]{Proposition}

\theoremstyle{definition}
\newtheorem{definition}[theorem]{Definition}

\theoremstyle{remark}

\theoremstyle{remark}

\begin{document}
 
\title{Pollyanna and Polynomially $\chi$-Bounded Graph Classes}
\author{
N. Rahimi$^1$\thanks{ORCID: 0009-0004-9571-669X}\ , D.A. Mojdeh$^2$\thanks{Corresponding author, ORCID: 0000-0001-9373-3390}\\
$^{1,2}$Department of Mathematics, Faculty of Mathematical Sciences,\\ University of Mazandaran, Babolsar, Iran\\
 $^{1}${\tt narjesrahimi1365@gmail.com}\\
 $^{2}${\tt damojdeh@umz.ac.ir}  
}
\date{}
\maketitle

\begin{abstract}

A hereditary graph class is called polynomially $\chi$-bounded if there exists a polynomial function $f$ such that $\chi(G) \le f(\omega(G))$ for every induced subgraph $G$. A class $\mathcal{C}$ is called Pollyanna if, for every $\chi$-bounded class $\mathcal{F}$, the class $\mathcal{C} \cap \mathcal{F}$ is polynomially $\chi$-bounded. 

In the paper Chudnovsky et al., \emph{Reuniting $\chi$-boundedness with polynomial $\chi$-boundedness} (J.\ Combin.\ Theory Ser.\ B 176 (2026), 30--73), the authors posed twelve problems and one conjecture concerning the Pollyanna framework. In this work, we investigate several of these problems by studying the chromatic number of hereditary graph classes defined by forbidden induced subgraphs.

We prove three new strong Pollyanna results. In particular, for every $t \ge 2$, every $\{\text{diamond}, \mathrm{hammer}(t)^+\}$-free graph is $(t)$-strongly Pollyanna. We also show that graph classes obtained by forbidding suitable combinations of bowties, dumbbells are $(2t-2)$-strongly Pollyanna.

We show that the class of $\{(2,2)\text{-bowtie},\, P_5,\, (3,3)\text{-dumbbell}\}$-free graphs is polynomially $\chi$-bounded. We also prove polynomial $\chi$-boundedness for diamond-free graphs in which every edge lies in at least two triangles, under additional forbidden configurations.

\end{abstract}

\section{Introduction}

All graphs considered in this paper are finite, simple, and undirected. For a positive integer $k$, a $k$-coloring of a graph $G$ is a function $c : V(G) \to \{1,\dots,k\}$ such that $c(u) \neq c(v)$ for every edge $uv \in E(G)$. A graph is called $k$-colorable if it admits a $k$-coloring. The chromatic number $\chi(G)$ is defined as the least integer $k$ for which $G$ is $k$-colorable. A clique in $G$ is a set of vertices that are pairwise adjacent. The size of a largest clique in $G$ is denoted by $\omega(G)$. Evidently, $\chi(H) \ge \omega(H)$ holds for every induced subgraph $H$ of $G$.

A class of graphs $\mathcal{G}$ is called \emph{hereditary} if every induced subgraph of any graph in $\mathcal{G}$ also belongs to $\mathcal{G}$. One important and well-studied hereditary graph class is the family of $H$-free graphs, i.e., graphs that have no induced subgraph isomorphic to a fixed graph $H$. Given a family of graphs $\mathcal{H}$, we say that a graph $G$ is \emph{$\mathcal{H}$-free} if $G$ is $H$-free for every $H \in \mathcal{H}$.

For certain graph families it is possible to bound the chromatic number purely in terms of the clique number. Following Gy\'arf\'as~\cite{gyarfas87}, such families are said to be $\chi$-bounded. More precisely, a family $\mathcal{G}$ is $\chi$-bounded if there exists a function $f : \mathbb{N} \to \mathbb{N}$, called a $\chi$-bounding function, such that for every graph $G \in \mathcal{G}$ and every induced subgraph $H$ of $G$, we have $\chi(H) \le f\bigl(\omega(H)\bigr)$.

Polynomial $\chi$-boundedness was a long-standing open problem originally proposed by Esperet \cite{Esperet2009}. This conjecture was recently disproven in a strong sense by Briański, Davies, and Walczak \cite{Brian2022}. Esperet's Conjecture suggests that every hereditary $\chi$-bounded class of graphs admits a $\chi$-bounding function that is a polynomial \cite{Esperet2009, Brian2022}.

While the conjecture initially seemed ambitious, it remained open for many years as several attempts to find a counterexample failed. Interestingly, research trends often supported the conjecture; many graph classes initially shown to be $\chi$-bounded were later proven to be polynomially $\chi$-bounded. Notable examples include classes of bounded broom-free graphs \cite{Kierstead1994, Liu2023}, and double-star-free graphs \cite{Kierstead1994, Scott2022}.

Furthermore, Esperet's Conjecture would have had significant implications for the Erdős–Hajnal property, implying that every hereditary $\chi$-bounded class satisfies it. A class $\mathcal{F}$ has this property if there exists $\epsilon > 0$ such that every $n$-vertex graph in $\mathcal{F}$ contains a clique or an independent set of size at least $n^\epsilon$ \cite{Erdos1989}. For a comprehensive history, see the survey by Chudnovsky \cite{Chudnovsky2014}. It remains an open question whether classes forbidding an induced path satisfy the Erdős–Hajnal property, despite significant recent progress \cite{Nguyen2023b, p5Scott2023, Nguyen2023}. Recently, Nguyen, Scott, and Seymour proved that classes with bounded VC-dimension satisfy this property \cite{Nguyen2024}.

In a graph $G$, the \emph{neighborhood} of a vertex $x$ is the set $N_G(x) = \{ y \in V(G) \setminus \{x\} \mid xy \in E(G)\}$.
Given two disjoint subsets $X, Y \subseteq V(G)$, we say that $X$ is \emph{complete} to $Y$ (or $[X, Y]$ is complete) if every vertex in $X$ is adjacent to every vertex in $Y$; 
and $X$ is \emph{anticomplete} to $Y$ if $[X, Y] = \emptyset$ (i.e., there are no edges between $X$ and $Y$).

For a positive integer~$i$, let  
$
N^{i}(X) := \{\, u \in V(G) \setminus X : \min\{ d(u,v) : v \in X \} = i \,\},
$
where $d(u,v)$ is the distance between $u$ and $v$ in~$G$.  
Then $N^{1}(X) = N(X)$ is the neighborhood of~$X$.  
Moreover, let
$
N^{\ge i}(X) := \bigcup_{j=i}^{\infty} N^{j}(X).
$
Let $H$ be any subgraph of~$G$.  
For simplicity, we sometimes write $N^{i}(H)$ for $N^{i}(V(H))$ and $N^{\ge i}(H)$ for $N^{\ge i}(V(H))$.

For $u,v \in V(G)$, we simply write $u \sim v$ if $uv \in E(G)$,  
and write $u \nsim v$ if $uv \notin E(G)$.
For any integer $\ell$, we let $P_\ell$ denote the path on $\ell$ vertices.
Gy\'arf\'as~\cite{gyarfas87} showed that the class of $P_k$-free graphs is $\chi$-bounded.

The \emph{diamond} is the graph obtained from the complete graph $K_4$ by deleting one edge.

In \cite{Cameron2021}, it is shown that every $(P_6, \text{diamond})$-free graph $G$ satisfies $\chi(G) \le \omega(G) + 3$. The kite is the graph obtained from a $P_4$ by adding a vertex and making it adjacent to all vertices in the $P_4$ except one vertex with degree 1. 
It is shown in \cite{Huang2024} that every $(P_5, \text{kite})$-free graph $G$ with $\omega(G) \ge 5$ satisfies $\chi(G) \le 2\omega(G) + 3$.

A flag is the graph obtained from a $K_4$ by attaching a pendent vertex. Char and Karthick \cite{Char2023} proved that for any $(P_5, \text{flag})$-free graph $G$ with $\omega(G) \ge 4$, $\chi(G) \le \linebreak \max\{8, 2\omega(G) - 3\}$.

In \cite{Chudnovsky2020}, the authors provide an explicit structural description of $(P_5, \text{gem})$-free graphs, and prove that every such graph $G$ satisfies $\chi(G) \le \left\lceil \frac{5\omega(G)}{4} \right\rceil$. Moreover, this bound is shown to be best possible. Here, a \emph{gem} is the graph consisting of an induced four-vertex path plus a vertex that is adjacent to all vertices of that path.

Chudnovsky et al. (2025) introduced the notion of a \emph{Pollyanna} class. Motivated by this,
the Pollyanna framework considers an internal analogue of the conjecture: a
class $\mathcal{C}$ is called \emph{Pollyanna} if $\mathcal{C} \cap \mathcal{F}$
is polynomially $\chi$-bounded for every $\chi$-bounded class $\mathcal{F}$ \cite{Chudnovsky2026reuniting}.
Equivalently, Esperet’s conjecture holds inside $\mathcal{C}$. Every
polynomially $\chi$-bounded class is Pollyanna, so Pollyanna classes strictly
generalize polynomially $\chi$-bounded ones.

 For an integer
$n$, let $\chi^{(n)}(G)$ denote the maximum chromatic number of an induced
subgraph $H$ of $G$ with $\omega(H)\le n$. A class $\mathcal{F}$ of graphs is
said to be \emph{$n$-good} if it is hereditary and there exists a constant $m$
such that every $G\in\mathcal{F}$ satisfies $\chi^{(n)}(G)\le m$. Note that
$n$-goodness is strictly weaker than $\chi$-boundedness.

A class $\mathcal{C}$ of graphs is called \emph{$n$-strongly Pollyanna} if,
for every $n$-good class $\mathcal{F}$ of graphs, the class $\mathcal{C}\cap
\mathcal{F}$ is polynomially $\chi$-bounded.

By considering the subclass $\mathcal{F}$ of $\mathcal{C}$ consisting of graphs $G$
with $\chi^{(n)}(G)\le c$, we can observe that $\mathcal{C}$ is $n$-strongly Pollyanna if and only if
there exists a function $f$ such that
\[
\chi(G) \le f\bigl(\chi^{(n)}(G),\, \omega(G)\bigr)
\]
for every graph $G \in \mathcal{C}$, where $f(c,\cdot)$ is a polynomial for
every constant $c$ \cite{Chudnovsky2026reuniting}.

\begin{definition}
The following graphs are defined as follows.
\begin{enumerate}
    \item Let $t$ and $k$ be positive integers. A \emph{$(t,k)$-pineapple} is a graph obtained
    by attaching $k$ pendant edges to a vertex of a complete graph $K_t$.

    \item Let $s$ and $t$ be positive integers. A graph is called an \emph{$(s,t)$-bowtie} if
    it can be obtained from the disjoint union of $K_s$ and $K_t$ by adding a new vertex that
    is adjacent to all other vertices.

    \item Let $t$ be a positive integer. The \emph{$t$-lollipop} is the graph obtained from the
    disjoint union of the complete graph $K_t$ and the path graph $P_2$ by adding an edge
    joining a vertex of $K_t$ to an end-vertex of $P_2$.
    
    \item Let $s,t \ge 1$ be integers. A graph is called an \emph{$(s,t)$-dumbbell} if it can be obtained from the disjoint union of the complete graphs $K_s$ and $K_t$ by adding a single edge between a vertex of $K_s$ and a vertex of $K_t$.

\end{enumerate}
\end{definition}

\begin{figure}[H]
\centering

\begin{subfigure}{0.4\textwidth}
\centering
\begin{tikzpicture}[scale=1.1,
    main/.style={circle, draw, fill=black, inner sep=1.5pt},
    every edge/.style={thick}
]

\node[main] (v1) at (0,0) {};

\node[main] (v2) at (-0.6,-0.4) {};
\node[main] (v3) at (-0.6,0.4) {};
\draw (v1)--(v2); \draw (v2)--(v3); \draw (v3)--(v1);

\node[main] (v4) at (0.6,-0.4) {};
\node[main] (v5) at (0.6,0.4) {};
\draw (v1)--(v4); \draw (v4)--(v5); \draw (v5)--(v1);

\end{tikzpicture}
\caption{$(2,2)$-bowtie}
\label{fig:bowtie22_horizontal}
\end{subfigure}
\hfill
\begin{subfigure}{0.45\textwidth}
\centering
\begin{tikzpicture}[scale=1,
    main/.style={circle, draw, fill=black, inner sep=1.5pt},
    every edge/.style={thick}
]

\node[main] (a1) at (-1,-0.5) {};
\node[main] (a2) at (-0.5,-1) {};
\node[main] (a3) at (0,-0.5) {};
\node[main] (a4) at (-0.5,0) {};

\draw (a1)--(a2); \draw (a1)--(a3); \draw (a1)--(a4);
\draw (a2)--(a3); \draw (a2)--(a4);
\draw (a3)--(a4);

\node[main] (b1) at (1,-0.5) {};
\node[main] (b2) at (1.5,0) {};
\node[main] (b3) at (1.5,-1) {};

\draw (b1)--(b2); \draw (b1)--(b3); \draw (b2)--(b3);

\draw (a3)--(b1);

\end{tikzpicture}
\caption{$(3,4)$-dumbbell}
\label{fig:dumbbell}
\end{subfigure}

\caption{}
\label{fig:bowtie-dumbbell22}
\end{figure}

\begin{theorem}[\cite{Chudnovsky2026reuniting}]\label{thm:1.2}
Let $m, k, t$ be positive integers. The following statements hold:
\begin{enumerate}
    \item The class of $mK_t$-free graphs is $(t-1)$-strongly Pollyanna.
    \item The class of $(t, k)$-pineapple-free graphs is $(2t-4)$-strongly Pollyanna.
    \item The class of $t$-lollipop-free graphs is $(3t-6)$-strongly Pollyanna.
    \item The class of bowtie-free graphs is $3$-strongly Pollyanna.
    \item The class of bull-free graphs is $2$-strongly Pollyanna.
\end{enumerate}
\end{theorem}

For positive integers $s$ and $t$, let $R(s, t)$ be the minimum positive integer $N$ such that
every graph on $N$ vertices contains a clique of size $s$ or an independent set of size $t$.

\begin{proposition}[\cite{Erdos1935}]\label{prop:Erdos1935}
For positive integers $s$ and $t$, we have
$
R(s, t) \le \binom{s + t - 2}{t - 1}.
$
\end{proposition}

Because of Proposition~\ref{prop:Erdos1935}, if $t$ is a fixed constant, then $R(s, t)$ is bounded from above by a degree-$(t-1)$ polynomial in $s$.

\section{Main results}
   
The structure is inspired by the theorem for pineapple-free graph classes \cite{Chudnovsky2026reuniting}.

\subsection{Structure}
Let $K \subseteq V(G)$ be a clique of size $\omega(G)$. Let $t$ be an integer. For $\emptyset \neq M \subseteq K$, $|M|<t$, define
\[
A_M = \{v \in V(G)\setminus K : v \text{ complete to } K\setminus M, \text{ anti-complete to } M\}.
\]

For a subset $N$ of $K$ with $|N| = t$ and a vertex $v$ of $K \setminus N$, let $A'_{N,v}$ be the set of vertices in
$N(v) \setminus K$ that are anti-complete to $N$.
Note that, by definition, every vertex $u \in N(K)$ with at least $t$ non-neighbors in $K$ is
in $A'_{N,v}$ for some $N \subseteq K \setminus N(u)$ and $v \in K$ with $|N| = t$.
Let $T$ be the union of all $A'_{N,v}$ for every choice of $N \subseteq K$ and $v \in K \setminus N$
such that $|N| = t$.

Then
\[
S = \bigcup_{1\le |M|<t} A_M, \quad T = \bigcup_{|N|=t} A'_{N,v}.
\]

It follows from these definitions that $S$ is the set of all vertices in $N(K)$ with fewer than $t$ non-neighbors in $K$, and $T$ is the set of all vertices in $N(K)$ with at least $t$ non-neighbors in $K$. Hence,
$N(K) = S \cup T$.

Let
\[
S' = \{ v \in V(G) \setminus (K \cup S \cup T) : v \text{ has a neighbor in } S \}
\]
\[
T' = \{ v \in V(G) \setminus (K \cup S \cup T) : v \text{ has a neighbor in } T \}
\]

{
\setlength{\itemsep}{0pt}
\setlength{\parskip}{0pt}
\setlength{\topsep}{3pt}

\begin{definition}
Let $t \ge 2$ and $k \ge 1$ be integers. Consider a complete graph $K_t$ on $t$ vertices and a star $S_k$ with $k$ vetices.  
A \emph{$(k,t)$-lollipop} is obtained by connecting the center of the star to every vertex of $K_t$ (Figure~\ref{fig:lollipop}).
  
\end{definition}
\vspace{-0.3\baselineskip}

\begin{definition}
Let $l$ be an integer.
A graph $F(3, l)$ consists of a central vertex adjacent to all vertices of $l$ disjoint triangles (i.e., $l$ triangles that are pairwise vertex-disjoint). In other words, it is obtained by taking $l$ copies of $K_3$ and connecting every vertex of each triangle to a single common central vertex (Figure~\ref{fig:F33}).

\end{definition}

\vspace{-0.3\baselineskip}

\begin{definition}
Let $t$ be an integer.
The graph $\mathrm{hammer}(t)^+$ is formed by a path of length~3, say $(u_0, u_1, u_2, u_3)$, where the endpoint $u_3$ is adjacent to all vertices of a complete graph $K_t$. In other words, it is the union of a $P_4$ and a $K_t$, joined by making one endpoint of the path adjacent to every vertex of the clique (Figure~\ref{fig:hamt+}).

\end{definition}

\vspace{-0.3\baselineskip}
\begin{definition}
Let $t \ge 1$ be an integer. Consider the complete graph $K_{t}$ on $t$ vertices.  
The graph $F^{1}_{t}$ is obtained by adding two new vertices, which are non-adjacent to each other, each of which is adjacent to all vertices of $K_{t}$ (Figure 2.2).
\end{definition}

\vspace{-0.3\baselineskip}

\begin{definition}
Let $t \ge 1$ be an integer. Consider the complete graph $K_t$ on $t$ vertices.  
The graph $F^{2}_{t}$ is obtained by adding three new vertices, which are pairwise non-adjacent, each of which is adjacent to all vertices of $K_t$ (Figure 2.2).
\end{definition}

}

\begin{figure}[H]
\centering

\begin{subfigure}{0.32\textwidth}
\centering
\begin{tikzpicture}[scale=1,
    main/.style={circle, draw, fill=black, inner sep=1.5pt},
    leaf/.style={circle, draw, fill=black, inner sep=1pt},
    every edge/.style={thick}
]

\node[main,green] (a1) at (-0.8,-1.5) {};
\node[main,green] (a2) at (0,-2) {};
\node[main,green] (a3) at (0.8,-1.5) {};
\node[main] (w) at (0,-0.7) {};
\draw (a1)--(a2); \draw (a2)--(a3); \draw (a3)--(a1);
\draw (w)--(a1); \draw (w)--(a2); \draw (w)--(a3);

\node[main] (o) at (0,0.5) {};
\draw (w)--(o);

\node[leaf,red] (l4) at (-1.1,1.1) {};
\node[leaf,red] (l1) at (-0.5,1.4) {};
\node[leaf,red] (l2) at (0.2,1.4) {};
\node[leaf,red] (l3) at (0.9,1.2) {};
\draw (o)--(l1); \draw (o)--(l2); \draw (o)--(l3);\draw (o)--(l4);
\end{tikzpicture}
\caption{$(4,3)$-lollipop}
\label{fig:lollipop}
\end{subfigure}
\hfill
\begin{subfigure}{0.32\textwidth}
\centering
\begin{tikzpicture}[
    scale=0.8,
    vertex/.style={circle, fill=black, inner sep=1.6pt},
    center/.style={circle, fill=black, inner sep=2.2pt},
    every edge/.style={thick}
]
\node[center] (c) at (0,0) {};
\def\r{2} \def\s{0.7}
\path (90:\r) coordinate (A); \path (210:\r) coordinate (B); \path (330:\r) coordinate (D);

\node[vertex] (a1) at ($(A)+(90:\s)$) {};
\node[vertex] (a2) at ($(A)+(210:\s)$) {};
\node[vertex] (a3) at ($(A)+(330:\s)$) {};
\draw[green] (a1)--(a2)--(a3)--(a1);

\node[vertex] (b1) at ($(B)+(90:\s)$) {};
\node[vertex] (b2) at ($(B)+(210:\s)$) {};
\node[vertex] (b3) at ($(B)+(330:\s)$) {};
\draw[green] (b1)--(b2)--(b3)--(b1);

\node[vertex] (d1) at ($(D)+(90:\s)$) {};
\node[vertex] (d2) at ($(D)+(210:\s)$) {};
\node[vertex] (d3) at ($(D)+(330:\s)$) {};
\draw[green] (d1)--(d2)--(d3)--(d1);

\foreach \v in {a1,a2,a3,b1,b2,b3,d1,d2,d3} {\draw (\v)--(c);}
\end{tikzpicture}
\caption{$F(3,3)$}
\label{fig:F33}
\end{subfigure}
\hfill
\begin{subfigure}{0.32\textwidth}
\centering
\begin{tikzpicture}[
    scale=0.6,
    vertex/.style={circle, fill=black, inner sep=1.6pt},
    every edge/.style={thick}
]
\node[vertex] (v1) at (0,3) {};
\node[vertex] (v2) at (0,2) {};
\node[vertex] (v3) at (0,1) {};
\node[vertex] (v4) at (0,0) {};
\draw (v1)--(v2)--(v3)--(v4);
\node[vertex,green] (a1) at (-1,-1.5) {};
\node[vertex,green] (a2) at (1,-1.5) {};
\node[vertex,green] (a3) at (0,-2.4) {};
\draw (a1)--(a2)--(a3)--(a1);
\foreach \v in {a1,a2,a3} {\draw (v4)--(\v);}
\end{tikzpicture}
\caption{$\mathrm{hammer}(3)^+$}
\label{fig:hamt+}
\end{subfigure}

\caption{}
\label{fig:special-graphs}
\end{figure}

\begin{figure}[H]
    \centering

    \begin{minipage}[c]{0.45\textwidth}
        \centering
        \begin{tikzpicture}[
            scale=1.1,
            vertex/.style={circle, fill=black, inner sep=1.2pt},
            f_node/.style={circle, fill=black, inner sep=1.8pt},
            set/.style={draw, thick, ellipse, minimum width=2.5cm, minimum height=1cm}
        ]

        \node[set] (K) at (0, 0) {\small $t$};

        \node[f_node] (f1) at (0, 1.5) {};
        \node[f_node] (f2) at (0, -1.5) {};

        \fill[gray!50, opacity=0.7] (f1) -- (K.west) -- (K.east) -- cycle;
        \fill[gray!50, opacity=0.7] (f2) -- (K.west) -- (K.east) -- cycle;

        \node[set] at (0, 0) {\small $t$};
        \node[f_node] at (0, 1.5) {};
        \node[f_node] at (0, -1.5) {};

        \end{tikzpicture}

        \caption*{$F^1_t$}
    \end{minipage}
    \hfill
    \begin{minipage}[c]{0.45\textwidth}
        \centering
        \begin{tikzpicture}[
            scale=1.1,
            vertex/.style={circle, fill=black, inner sep=1.2pt},
            f_node/.style={circle, fill=black, inner sep=1.8pt},
            set/.style={draw, thick, ellipse, minimum width=2.5cm, minimum height=1cm}
        ]

        \node[set] (K) at (0,0) {\small $t$};

        \node[f_node] (f1) at (-1.6,1.6) {};
        \node[f_node] (f2) at (1.6,1.6) {};
        \node[f_node] (f3) at (0,-1.7) {};

        \fill[gray!50, opacity=0.7] (f1) -- (K.west) -- (K.east) -- cycle;
        \fill[gray!50, opacity=0.7] (f2) -- (K.west) -- (K.east) -- cycle;
        \fill[gray!50, opacity=0.7] (f3) -- (K.west) -- (K.east) -- cycle;

        \node[set] at (0,0) {\small $t$};
        \node[f_node] at (-1.6,1.6) {};
        \node[f_node] at (1.6,1.6) {};
        \node[f_node] at (0,-1.7) {};

        \end{tikzpicture}

        \caption*{$F^2_t$}
    \end{minipage}

    \caption{$F^1_t$ and $F^2_t$.}
\end{figure}

For integers $s,t \ge 2$ with $s \le t$, and for a positive integer $C$, if $|K| \ge 2t - 1$, then we have the following properties:

\textit{P-property.}
A graph $G$ is said to have the P-property if for every induced subgraph 
$G' \subseteq G$ with $\omega(G') \le t$, it holds that $\chi(G') \le C$.

\medskip

\textit{Property 1.}

If $G$ is $F^1_t$-free, then $S = \varnothing$. In particular, when $t = 2$ (so that the forbidden graph is the diamond), we have $S = \varnothing$ even for graphs with $\omega(G) \ge 3$ (Figure~\ref{fig:P1P2}a).

\medskip

\textit{Property 2.}

If $G$ is $F^2_t$-free, then either $|A_M| = \{0,1\}$ or if two vertices lie in $A_M$, they must be adjacent. 
Therefore,
$|A_M| \le \omega(G)$ and consequently
$
|S| \le \sum_{i=1}^{t-1} \omega(G) \binom{\omega(G)}{i}
$ (Figure~\ref{fig:P1P2}b).

\begin{figure}[H]
\centering

\begin{subfigure}{0.48\textwidth}
\centering
\scalebox{0.7}{
\begin{tikzpicture}[scale=0.8,
    big_set/.style={draw=black, thick, ellipse, minimum width=6cm, minimum height=3cm},
    inner_set/.style={draw=black, thick, ellipse, minimum width=3.5cm, minimum height=1.8cm},
    vertex/.style={circle, fill=black, inner sep=2pt}
]
\node[big_set, label=left:$S$] (S) at (0, 4) {};
\node[inner_set, label=left:$A_M$, color=red!80] (AM) at (0, 4) {};
\node[vertex] (V1) at (-1, 4.2) {};
\node[big_set, label=left:$K$] (K) at (0, -1) {};
\coordinate (curve_top) at (-1, 0.8);
\coordinate (curve_bottom) at (-1, -2.8);
\draw[dashed, thick, color=red!80] 
    (-1, 0.8) .. controls (-1.5, 0) and (-1.5, -2) .. (-1, -2.8);
\node[vertex] (V2) at (-2.3, -0.8) {};
\node[vertex] (W1) at (1.2, 0.2) {};
\node[vertex] (W2) at (1.2, -1.8) {};
\node[font=\small, color=red!80] at (-2.3, -2) {$M$};
\fill[blue!20, opacity=0.6] (V1.center) -- (W1.center) -- (W2.center) -- cycle;
\fill[blue!20, opacity=0.6] (V2.center) -- (W1.center) -- (W2.center) -- cycle;
\draw (V1) -- (W1);
\draw (V1) -- (W2);
\draw (V2) -- (W1);
\draw (V2) -- (W2);
\draw (W1) -- (W2);
\end{tikzpicture}
}
\caption{$P1$}
\label{fig:P1}
\end{subfigure}
\hfill
\begin{subfigure}{0.48\textwidth}
\centering
\scalebox{0.7}{
\begin{tikzpicture}[scale=0.8,
    big_set/.style={draw=black, thick, ellipse, minimum width=6cm, minimum height=3cm},
    inner_set/.style={draw=black, thick, ellipse, minimum width=3.5cm, minimum height=1.8cm},
    vertex/.style={circle, fill=black, inner sep=2pt}
]
\node[big_set, label=left:$S$] (S) at (0, 4) {};
\node[inner_set, label=left:$A_M$, color=red!80] (AM) at (0, 4) {};
\node[vertex] (V1) at (-1.3, 4.2) {};
\node[vertex] (V3) at (-0.5, 4.2) {};
\node[big_set, label=left:$K$] (K) at (0, -1) {};
\coordinate (curve_top) at (-1, 0.5);
\coordinate (curve_bottom) at (-1, -2.5);
\draw[dashed, thick, color=red!80] 
    (-1, 0.8) .. controls (-1.5, 0) and (-1.5, -2) .. (-1, -2.8);
\node[vertex] (V2) at (-2.3, -0.8) {};
\node[vertex] (W1) at (1.6, -0.2) {};
\node[vertex] (W2) at (1.2, -1.8) {};
\node[font=\small, color=red!80] at (-2.3, -1.8) {$M$};
\fill[blue!20, opacity=0.6] (V1.center) -- (W1.center) -- (W2.center) -- cycle;
\fill[blue!20, opacity=0.6] (V2.center) -- (W1.center) -- (W2.center) -- cycle;
\fill[green!20, opacity=0.6] (V3.center) -- (W1.center) -- (W2.center) -- cycle;
\draw (V1) -- (W1);
\draw (V1) -- (W2);
\draw (V2) -- (W1);
\draw (V2) -- (W2);
\draw (V3) -- (W1);
\draw (V3) -- (W2);
\draw (W1) -- (W2);
\end{tikzpicture}
}
\caption{$P2$}
\label{fig:P2}
\end{subfigure}

\caption{}
\label{fig:P1P2}
\end{figure}

\textit{Property 3.}

Let $k$ be an integer. If the graph is $(k,t)$-lollipop-free, then For each vertex $v_0\in T$, we have
\[
|N_{T'}(v_0)| < R(\omega-1, k).
\]

\medskip
\begin{figure}[H]
\centering
\scalebox{0.8}{
\begin{tikzpicture}[
    big_set/.style={draw=black, thick, ellipse, minimum width=6cm, minimum height=2.5cm},
    inner_set/.style={draw=black, thick, ellipse, minimum width=3cm, minimum height=1.2cm},
    vertical_ellipse/.style={draw=black, thick, ellipse, minimum width=0.8cm, minimum height=1.5cm, fill=gray!40},
    vertex/.style={circle, fill=black, inner sep=1.8pt},
    vertex_red/.style={circle, fill=red!80, inner sep=1.8pt},
    green_edge/.style={draw=green!60!black, thick},
    green_cone/.style={fill=green!40, opacity=0.6},
    font=\small
]

\node[big_set, label=left:$T'$] (Tprime) at (0, 5) {};

\node[vertex] (V1) at (-1.5, 4.8) {};
\node[vertex] (V2) at (0, 4.8) {};
\node[vertex] (V3) at (1.5, 4.8) {};

\node[big_set, label=left:$T$] (T) at (0, 2) {};

\node[inner_set] (ANv) at (0, 2) {};

\node[left=0cm, font=\small] at (ANv.west) {$A'_{N,v}$};

\coordinate (ANv_left) at (ANv.west);
\coordinate (ANv_right) at (ANv.east);

\node[big_set, label=left:$K$] (K) at (0, -1) {};

\coordinate (divider_mid) at ($(K.west)!0.33!(K.east)$);
\coordinate (divider_top) at (divider_mid |- K.north);
\coordinate (divider_bottom) at (divider_mid |- K.south);

\draw[dashed, thick, red!80] 
    (divider_top) .. controls ($(divider_top) + (-0.3, 0)$) and ($(divider_bottom) + (-0.3, 0)$) .. 
    (divider_bottom);

\node[font=\small, color=red!80] at ($(divider_top)!0.5!(divider_bottom) + (1.5, -1)$) {$N$};

\node[vertex, label=left:$v$] (v) at ($(K.west)!0.5!(divider_mid) + (0, 0.2)$) {};

\coordinate (t_ellipse_center) at (1.2, -1);

\node[vertical_ellipse, fill=green!25] (t_ellipse) at (t_ellipse_center) {};
\node[font=\small] at (t_ellipse.center) {$t$};

\coordinate (t_ellipse_top) at (t_ellipse.north);
\coordinate (t_ellipse_bottom) at (t_ellipse.south);

\fill[gray!50, opacity=0.5] (v.center) -- (ANv_left) -- (ANv_right) -- cycle;

\node[inner_set, draw=black, thick, fill=white] at (0, 2) {};

\node[vertex_red] (D) at (-0.5, 2.2) {};

\fill[green!40, opacity=0.6] (v.center) -- (t_ellipse_top) -- (t_ellipse_bottom) -- cycle;

\node[vertical_ellipse, draw=black, thick, fill=green!25] at (t_ellipse_center) {};
\node[font=\small] at (t_ellipse.center) {$t$};

\draw[green_edge] (V1) -- (D);
\draw[green_edge] (V2) -- (D);
\draw[green_edge] (V3) -- (D);

\draw[green_edge, line width=1.5pt] (D) -- (v);

\draw[green_edge] (v) -- (t_ellipse_top);
\draw[green_edge] (v) -- (t_ellipse_bottom);
\node[vertex_red, label=left:$v_0$] (D) at (-0.5, 2.2) {};
\end{tikzpicture}
}
\caption{$P3$ (for \emph{$(3,t)$-lollipop})}
\label{fig:P3}
\end{figure}
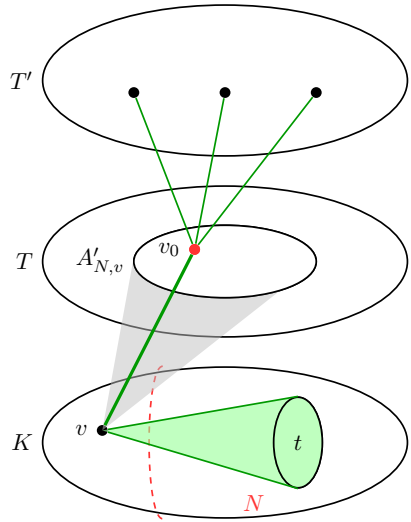

\textit{Property 4.}

If the graph is $\{hammer(t)^+, \text{diamond}\}$-free, then $\chi(T') \le \omega(G)$.

\medskip

Since the graph $G$ is $\text{hammer}(t)^+$-free, for every $v_0 \in T$ we have
\[
N^{\ge 2}(v_0) \setminus (K \cup T) = \emptyset.
\]
Indeed, if there exists a vertex $O \in N^{\ge 2}(v_0) \setminus (K \cup T)$, then an induced
$\text{hammer}(t)^+$ would appear, which is a contradiction (Figure~\ref{fig:both-P4}a).

Therefore, every vertex in $T$ must be adjacent to all vertices of $T'$.
If $T'$ is a complete subgraph, or consists of components each of which is a complete subgraph,
then the claim holds.

Hence, assume to the contrary that $T'$ contains a component $T''$ which is not complete.
Then $T''$ contains a clique $K''$ of size at least two.
Let $k_1, k_2 \in K''$ (in Figure~\ref{fig:both-P4}b, we assume $|K''| = 3$).

Since $T''$ is assumed not to be complete, there must exist a vertex $v_1 \in T''$ that is at distance $2$ from some vertex of $K''$, say $k_2$.
By the above argument, every vertex $v_0 \in T$ is adjacent to all vertices of $T''$.
Consequently, the vertices $v_1, k_1, k_2, v_0$ induce a diamond, unless $v_1$ is adjacent
to $k_2$, which contradicts the assumption so $|T''| \leq \omega(G)$.

\begin{figure}[H]
\centering

\begin{minipage}{0.48\textwidth}
\centering
\scalebox{0.8}{
\begin{tikzpicture}[
    big_set/.style={draw=black, thick, ellipse, minimum width=6cm, minimum height=2.5cm},
    inner_set/.style={draw=black, thick, ellipse, minimum width=3cm, minimum height=1.2cm},
    vertical_ellipse/.style={draw=black, thick, ellipse, minimum width=0.8cm, minimum height=1.5cm, fill=green!25},
    vertex/.style={circle, fill=black, inner sep=1.8pt},
    vertex_red/.style={circle, fill=red!80, inner sep=1.8pt},
    green_edge/.style={draw=green!60!black, thick},
    dashed_line/.style={dashed, thick, black!60},
    font=\small
]

\coordinate (dashed_line_left) at (-4, 6.5);
\coordinate (dashed_line_right) at (4, 6.5);
\draw[dashed_line] (dashed_line_left) -- (dashed_line_right);
\node[vertex, label=above:$O$] (O) at (-0.7, 7.2) {};
\node[big_set, label=left:$T'$] (Tprime) at (0, 5) {};
\node[vertex] (V1) at (-1, 5.2) {};
\node[big_set, label=left:$T$] (T) at (0, 2) {};
\node[inner_set] (ANv) at (0, 2) {};
\node[left=0cm, font=\small] at (ANv.west) {$A'_{N,v}$};
\coordinate (ANv_left) at (ANv.west);
\coordinate (ANv_right) at (ANv.east);
\node[big_set, label=left:$K$] (K) at (0, -1) {};
\coordinate (divider_mid) at ($(K.west)!0.33!(K.east)$);
\coordinate (divider_top) at (divider_mid |- K.north);
\coordinate (divider_bottom) at (divider_mid |- K.south);
\draw[dashed, thick, red!80] 
    (divider_top) .. controls ($(divider_top) + (-0.3, 0)$) and ($(divider_bottom) + (-0.3, 0)$) .. 
    (divider_bottom);
\node[font=\small, color=red!80] at ($(divider_top)!0.5!(divider_bottom) + (1.5, -1)$) {$N$};
\node[vertex, label=left:$v$] (v) at ($(K.west)!0.5!(divider_mid) + (0, 0.2)$) {};
\coordinate (t_ellipse_center) at (1.2, -1);
\node[vertical_ellipse] (t_ellipse) at (t_ellipse_center) {};
\node[font=\small] at (t_ellipse.center) {$t$};
\coordinate (t_ellipse_top) at (t_ellipse.north);
\coordinate (t_ellipse_bottom) at (t_ellipse.south);
\fill[gray!50, opacity=0.5] (v.center) -- (ANv_left) -- (ANv_right) -- cycle;
\node[inner_set, draw=black, thick, fill=white] at (0, 2) {};
\node[vertex_red] (D) at (-0.5, 2.2) {};
\fill[green!40, opacity=0.6] (v.center) -- (t_ellipse_top) -- (t_ellipse_bottom) -- cycle;
\node[vertical_ellipse, draw=black, thick, fill=green!25] at (t_ellipse_center) {};
\node[font=\small] at (t_ellipse.center) {$t$};
\draw[green_edge] (D) -- (V1);
\draw[green_edge] (V1) -- (O);
\draw[green_edge, line width=1.5pt] (D) -- (v);
\draw[green_edge] (v) -- (t_ellipse_top);
\draw[green_edge] (v) -- (t_ellipse_bottom);
\node[above left, font=\small] at (dashed_line_right) {$N^{\ge 2}(v_0) \setminus (K \cup T) = \emptyset.$};
\node[vertex_red, label=left:$v_0$] (D) at (-0.5, 2.2) {};

\end{tikzpicture}
}
\caption*{(a)}
\end{minipage}
\hfill 
\begin{minipage}{0.48\textwidth}
\centering
\scalebox{0.8}{
\begin{tikzpicture}[
    big_set/.style={draw=black, thick, ellipse, minimum width=6cm, minimum height=3cm},
    vertex/.style={circle, fill=black, inner sep=1.8pt},
    vertex_blue/.style={circle, fill=blue!80, inner sep=1.8pt},
    vertex_red/.style={circle, fill=red!80, inner sep=1.8pt, minimum size=5pt},
    green_edge/.style={draw=green!70!black, thick},
    font=\small
]

\node[big_set, label={[font=\normalsize]above:$T''$}] (T_zagond) at (0, 3) {};
\coordinate (k1) at (-1, 2.8);
\coordinate (k2) at (1, 2.8);
\coordinate (k3) at (0, 3.8);
\draw[black!70, dashed, thick, rounded corners=10pt] 
    (-1.7, 2.5) -- (-1.7, 4) -- (1.7, 4) -- (1.7, 2.5) -- cycle;
\node[right=0.4cm, font=\small] at (1.7, 3.2) {$K''$};
\node[vertex, label=above left:$k_1$] (K1) at (k1) {};
\node[vertex, label=above right:$k_2$] (K2) at (k2) {};
\node[vertex] (K3) at (k3) {};
\fill[gray!25] (K1.center) -- (K2.center) -- (K3.center) -- cycle;
\draw[thick] (K1) -- (K2) -- (K3) -- (K1);
\node[vertex_blue, label=left:$v_1$] (V1) at (-2.2, 3) {};
\node[vertex_red, label=below:$v_0$] (V0) at (0, -0.5) {};
\draw[green_edge] (V0) to[bend left=5] (K1);
\draw[green_edge] (V0) to[bend right=5] (K2);
\draw[green_edge] (V0) -- (V1);
\draw[green_edge, line width=1.5pt] (V1) to[bend left=10] (K1);

\end{tikzpicture}
}
\caption*{(b)}
\end{minipage}

\caption{}
\label{fig:both-P4}
\end{figure}

\textit{Property 5.}

Suppose $G$ is \emph{$(s,t)$-bowtie}-free and satisfies the P-property.
$\omega(A'_{N,v}) < t$; since otherwise, a \emph{$(s,t)$-bowtie} would exist.

Given that the number of possible choices for $v$ and $N$ equals $\omega \binom{\omega-1}{t}$.
Then 
$\chi(T) \le C \omega \binom{\omega-1}{t}$.
In particular, if $t=2$, without the P-property, even for graphs with $\omega(G) \ge 3$, no edge can be found in $A'_{N,v}$, and edges may only exist between the sets of $A'_{N,v}$. Thus 
$\chi(T) \le \omega \binom{\omega-1}{2}$.

\begin{figure}[H]
\centering
\scalebox{0.7}{
\begin{tikzpicture}[
    big_set/.style={draw=black, thick, ellipse, minimum width=6cm, minimum height=2.8cm},
    dashed_set/.style={draw=black, thick, dashed, ellipse, minimum width=3.5cm, minimum height=1.5cm},
    omega_set/.style={draw=black, thick, ellipse, minimum width=2.5cm, minimum height=0.8cm, fill=green!25},
    vertical_ellipse/.style={draw=black, thick, ellipse, minimum width=0.9cm, minimum height=1.6cm, fill=green!30},
    vertex/.style={circle, fill=black, inner sep=1.8pt},
    green_edge/.style={draw=green!70!black, thick},
    font=\small
]

\node[big_set, label=left:$T$] (T) at (0, 3.5) {};

\node[dashed_set] (ANv) at (0, 3.5) {};
\node[left=0.1cm, font=\footnotesize, color=gray] at (ANv.west) {$A'_{N,v}$}; 

\node[omega_set] (omega_ANv) at (0, 3.5) {};
\node[font=\footnotesize] at (omega_ANv.center) {$\omega(A_{N,v})=t$};

\coordinate (omega_left) at (omega_ANv.west);
\coordinate (omega_right) at (omega_ANv.east);

\node[big_set, label=left:$K$] (K) at (0, -1) {};

\coordinate (divider_mid) at ($(K.west)!0.33!(K.east)$);
\coordinate (divider_top) at (divider_mid |- K.north);
\coordinate (divider_bottom) at (divider_mid |- K.south);

\draw[dashed, thick, red!80] 
    (divider_top) .. controls ($(divider_top) + (-0.3, 0)$) and ($(divider_bottom) + (-0.3, 0)$) .. 
    (divider_bottom);

\node[font=\small, color=red!80] at ($(divider_top)!0.5!(divider_bottom) + (1.5, -1)$) {$N$};

\node[vertex, label=left:$v$] (v) at ($(K.west)!0.5!(divider_mid) + (0, 0.2)$) {};

\coordinate (t_ellipse_center) at (1.2, -1);
\node[vertical_ellipse] (t_ellipse) at (t_ellipse_center) {};
\node[font=\small] at (t_ellipse.center) {$t$};

\coordinate (t_ellipse_top) at (t_ellipse.north);
\coordinate (t_ellipse_bottom) at (t_ellipse.south);

\fill[green!40, opacity=0.6] (v.center) -- (omega_left) -- (omega_right) -- cycle;

\fill[green!50, opacity=0.7] (v.center) -- (t_ellipse_top) -- (t_ellipse_bottom) -- cycle;

\draw[green_edge] (v) -- (t_ellipse_top);
\draw[green_edge] (v) -- (t_ellipse_bottom);
\draw[green_edge] (v) -- (omega_left);
\draw[green_edge] (v) -- (omega_right);

\node[vertical_ellipse, draw=black, thick, fill=green!30] at (t_ellipse_center) {};
\node[font=\small] at (t_ellipse.center) {$t$};

\node[dashed_set] (ANv_redraw) at (0, 3.5) {};
\node[omega_set, draw=black, thick, fill=green!25] at (0, 3.5) {};
\node[font=\footnotesize] at (omega_ANv.center) {$\omega(A'_{N,v})=t$};

\node[left=0.1cm, font=\footnotesize] at (ANv.west) {$A'_{N,v}$};

\end{tikzpicture}
}
\caption{$P5$}
\label{fig:P3_double_prime}
\end{figure}

\textit{Property 6.}

Suppose $G$ is $\{P_5, \emph{$(s,t)$-bowtie}\}$-free and satisfies the P-property. For every $v_0 \in S$ we have
$
N^{\ge 2}(v_0) \setminus (K \cup S) = \emptyset.
$
Otherwise, a $P_5$ would exist. Therefore, every vertex in $S$ must be adjacent to all vertices of $S'$. Then $\omega(S') < t $, since otherwise, a \emph{$(s,t)$-bowtie} would exist. Therefore 
$\chi(S') \le C$.
In the special case for $t=2$ and without the P-property, even for graphs with $\omega(G) \ge 3$, no edge exists between the vertices of $S'$, hence 
$\chi(S') \le 1$.
\medskip

\begin{figure}[H]
\centering
\scalebox{0.8}{
\begin{tikzpicture}[
    big_set/.style={draw=black, thick, ellipse, minimum width=6cm, minimum height=2.5cm},
    inner_set/.style={draw=black, thick, ellipse, minimum width=3cm, minimum height=1.2cm},
    component_ellipse/.style={draw=black, thick, ellipse, minimum width=2.5cm, minimum height=0.8cm, fill=green!20},
    vertical_ellipse/.style={draw=black, thick, ellipse, minimum width=0.8cm, minimum height=1.5cm, fill=green!25},
    vertex/.style={circle, fill=black, inner sep=1.8pt},
    vertex_red/.style={circle, fill=red!80, inner sep=1.8pt},
    green_edge/.style={draw=green!60!black, thick},
    dashed_line/.style={dashed, thick, black!60},
    rounded_rect/.style={draw=black, thick, dashed, rounded corners=6pt, minimum width=3.5cm, minimum height=1cm},
    font=\small
]

\coordinate (dashed_line_left) at (-4, 6.5);
\coordinate (dashed_line_right) at (4, 6.5);
\draw[dashed_line] (dashed_line_left) -- (dashed_line_right);

\node[big_set, label=left:$S'$] (Sprime) at (0, 5) {};

\node[rounded_rect] (component) at (0, 5) {};
\node[above=0.1cm, font=\small] at (component.north) {component B};

\node[component_ellipse] (inner_component) at (0, 5) {};
\node[font=\footnotesize] at (inner_component.center) {$\omega(B)=t$};

\path (inner_component.180) coordinate (omega_left_edge);   
\path (inner_component.0) coordinate (omega_right_edge);

\node[big_set, label=left:$S$] (S) at (0, 2) {};

\node[inner_set] (AM) at (0, 2) {};

\node[left=0cm, font=\small] at (AM.west) {$A_M$};

\node[big_set, label=left:$K$] (K) at (0, -1) {};

\coordinate (divider_mid) at ($(K.west)!0.33!(K.east)$);
\coordinate (divider_top) at (divider_mid |- K.north);
\coordinate (divider_bottom) at (divider_mid |- K.south);

\draw[dashed, thick, red!80] 
    (divider_top) .. controls ($(divider_top) + (-0.3, 0)$) and ($(divider_bottom) + (-0.3, 0)$) .. 
    (divider_bottom);

\node[font=\small, color=red!80] at ($(K.west)!0.5!(divider_mid) + (0, -0.5)$) {$M$};

\node[vertex] (v) at ($(K.west)!0.5!(divider_mid) + (0, 0.2)$) {};

\node[vertex_red, label=below right:$v_0$] (v0) at (-0.5, 2.2) {};

\coordinate (t_ellipse_center) at (1.2, -1);

\node[vertical_ellipse] (t_ellipse) at (t_ellipse_center) {};
\node[font=\small] at (t_ellipse.center) {$t$};

\coordinate (t_ellipse_top) at (t_ellipse.north);
\coordinate (t_ellipse_bottom) at (t_ellipse.south);

\fill[green!30, opacity=0.5] (v0) -- (omega_left_edge) -- (omega_right_edge) -- cycle;

\fill[gray!50, opacity=0.6] (v.center) -- (t_ellipse_top) -- (t_ellipse_bottom) -- cycle;

\fill[green!30, opacity=0.5] (v0) -- (t_ellipse_top) -- (t_ellipse_bottom) -- cycle;

\node[vertical_ellipse, draw=black, thick, fill=green!25] at (t_ellipse_center) {};
\node[font=\small] at (t_ellipse.center) {$t$};

\node[inner_set, draw=black, thick, fill=white] at (0, 2) {};

\node[component_ellipse, draw=black, thick, fill=green!20] at (0, 5) {};
\node[font=\footnotesize] at (inner_component.center) {$\omega(B)=t$};

\node[rounded_rect, draw=black, thick, dashed] at (0, 5) {};

\draw[green_edge] (v0) -- (omega_left_edge);
\draw[green_edge] (v0) -- (omega_right_edge);

\draw[green_edge] (v) -- (t_ellipse_top);
\draw[green_edge] (v) -- (t_ellipse_bottom);

\draw[green!60!black, thick] (v0) -- (t_ellipse_top);
\draw[green!60!black, thick] (v0) -- (t_ellipse_bottom);

\node[above left, font=\small] at (dashed_line_right) {$N^{\ge 2}(v_0) \setminus (K \cup S) = \emptyset.$};
\node[vertex_red, label=below right:$v_0$] (D) at (-0.5, 2.2) {}; 
\end{tikzpicture}
}
\caption{$P6$}
\label{fig:P6}
\end{figure}

\textit{Property 7.}

$G$ is $\{P_5, \emph{$(s+1,t+1)$-dumbbell}\}$-free and satisfies the P-property, then $\omega(T') < t $. For every $v_0 \in T$ we have
$
N^{\ge 2}(v_0) \setminus (K \cup T) = \emptyset.
$
Otherwise, a $P_5$ would exist. then $\omega(T') < t $ Therefore 
$\chi(T') \le C$, since otherwise, a \emph{$(s+1,t+1)$-dumbbell} would exist.
In the special case for $s,t=2$ without the P-property, no edge can exist in $T'$, as otherwise a \emph{$(3,3)$-dumbbell} would be present for graphs with $\omega(G) \ge 3$. Hence 
$\chi(T') \le 1$.

\begin{figure}[H]
\centering
\scalebox{0.8}{
\begin{tikzpicture}[
    big_set/.style={draw=black, thick, ellipse, minimum width=6cm, minimum height=2.5cm},
    inner_set/.style={draw=black, thick, ellipse, minimum width=3cm, minimum height=1.2cm},
    component_ellipse/.style={draw=black, thick, ellipse, minimum width=2.5cm, minimum height=0.8cm, fill=green!20}, 
    vertical_ellipse/.style={draw=black, thick, ellipse, minimum width=0.8cm, minimum height=1.5cm, fill=green!25},
    vertex/.style={circle, fill=black, inner sep=1.8pt},
    vertex_red/.style={circle, fill=red!80, inner sep=1.8pt},
    green_edge/.style={draw=green!60!black, thick},
    dashed_line/.style={dashed, thick, black!60},
    rounded_rect/.style={draw=black, thick, dashed, rounded corners=6pt, minimum width=3.5cm, minimum height=1cm},
    font=\small
]

\coordinate (dashed_line_left) at (-4, 6.5);
\coordinate (dashed_line_right) at (4, 6.5);
\draw[dashed_line] (dashed_line_left) -- (dashed_line_right);

\node[big_set, label=left:$T'$] (Tprime) at (0, 5) {};

\node[rounded_rect] (component) at (0, 5) {};
\node[above=0.1cm, font=\small] at (component.north) {component A};

\node[component_ellipse] (inner_component) at (0, 5) {};
\node[font=\footnotesize] at (inner_component.center) {$\omega(A)=t$};

\path (inner_component.180) coordinate (omega_left_edge);    
\path (inner_component.0) coordinate (omega_right_edge);     

\node[big_set, label=left:$T$] (T) at (0, 2) {};

\node[inner_set] (ANv) at (0, 2) {};

\node[left=0cm, font=\small] at (ANv.west) {$A'_{N,v}$};

\coordinate (ANv_left) at (ANv.west);
\coordinate (ANv_right) at (ANv.east);

\node[big_set, label=left:$K$] (K) at (0, -1) {};

\coordinate (divider_mid) at ($(K.west)!0.33!(K.east)$);
\coordinate (divider_top) at (divider_mid |- K.north);
\coordinate (divider_bottom) at (divider_mid |- K.south);

\draw[dashed, thick, red!80] 
    (divider_top) .. controls ($(divider_top) + (-0.3, 0)$) and ($(divider_bottom) + (-0.3, 0)$) .. 
    (divider_bottom);

\node[font=\small, color=red!80] at ($(divider_top)!0.5!(divider_bottom) + (1.5, -1)$) {$N$};

\node[vertex, label=left:$v$] (v) at ($(K.west)!0.5!(divider_mid) + (0, 0.2)$) {};

\coordinate (t_ellipse_center) at (1.2, -1);

\node[vertical_ellipse] (t_ellipse) at (t_ellipse_center) {};
\node[font=\small] at (t_ellipse.center) {$t$};

\coordinate (t_ellipse_top) at (t_ellipse.north);
\coordinate (t_ellipse_bottom) at (t_ellipse.south);

\fill[gray!50, opacity=0.5] (v.center) -- (ANv_left) -- (ANv_right) -- cycle;

\fill[green!30, opacity=0.5] (D) -- (omega_left_edge) -- (omega_right_edge) -- cycle;

\fill[green!40, opacity=0.6] (v.center) -- (t_ellipse_top) -- (t_ellipse_bottom) -- cycle;

\node[vertical_ellipse, draw=black, thick, fill=green!25] at (t_ellipse_center) {};
\node[font=\small] at (t_ellipse.center) {$t$};

\node[inner_set, draw=black, thick, fill=white] at (0, 2) {};

\node[component_ellipse, draw=black, thick, fill=green!20] at (0, 5) {};
\node[font=\footnotesize] at (inner_component.center) {$\omega(A)=t$};

\node[rounded_rect, draw=black, thick, dashed] at (0, 5) {};

\draw[green_edge] (D) -- (omega_left_edge);
\draw[green_edge] (D) -- (omega_right_edge);

\draw[green_edge, line width=1.5pt] (D) -- (v);

\draw[green_edge] (v) -- (t_ellipse_top);
\draw[green_edge] (v) -- (t_ellipse_bottom);
\node[vertex_red, label=below right:$v_0$] (D) at (-0.5, 2.2) {};
\node[above left, font=\small] at (dashed_line_right) {$N^{\ge 2}(v_0) \setminus (K \cup T) = \emptyset.$};

\end{tikzpicture}
}
\caption{$P7$}
\label{fig:P4}
\end{figure}
\textit{Property 8.}

Suppose that $G$ is diamond-free.
As in the proof of Property 4 (Figure~\ref{fig:both-P4}b), for each $v \in K$, due to the diamond-free property, If $T''$ is a component of $A'_{N,v}$, then $|T''| \le \omega(G)$. so
$
\chi(A'_{N,v}) \le \omega(G),
$
and the number of sets $A'_{N,v}$ satisfies
$
\omega \binom{\omega-1}{t}.
$
Hence, we obtain
$
\chi(T) \le \omega^2 \binom{\omega-1}{t}.
$
\\

In this paper, we consider the graph $Y$ to be either the graph $F^1_t$ or $F^2_t$.\\

\begin{theorem}\label{thm1}
Let $t \ge 2$ be an integer. Suppose that $G$ is a $\{\text{diamond}, \text{hammer}(t)^+\}$-free
 graph. Then $G$ is $(t)$-strongly Pollyanna.
\end{theorem}

\begin{proof}
Let $F$ be a hereditary class of graphs and let $C$ be a positive integer
such that $\chi(G) \le C$ whenever $G \in F$ and $\omega(G) \le t$.
Let $\mathcal{G}$ be the class of $\{\text{diamond}, \text{hammer}(t)^+\}$-free graphs and let $G \in F \cap \mathcal{G}$.
We may assume that $\omega(G)\ge t+1$, because otherwise $\chi(G)\le C$.
By Property 1, let $S = \emptyset$.
 According to Property 8, we have
\[
\chi(T) \le \omega^2 \binom{\omega-1}{t}.
\]

Also, by Property 4,
\[
\chi(T') \le \omega(G).
\]

Therefore,
\[
\chi(G) \le 2\omega(G) + \omega^2 \binom{\omega-1}{t}+C.
\]
\end{proof}

\newpage
\begin{theorem}\label{thm2}
Let $s,t,k\geq2$ be positive integers. Suppose that $G$ is a $\{Y,(s,t)\text{-bowtie},(k,t)\text{-lollipop}\}$-free graph. Then $G$ is $(2t-2)$-strongly Pollyanna.
\end{theorem}

\begin{proof}
Let $F$ be a hereditary class of graphs and let $C$ be a positive integer
such that $\chi(G) \le C$ whenever $G \in F$ and $\omega(G) \le 2t-2$.
Let $\mathcal{G}$ be the class of $\{Y,(s,t)\text{-bowtie},(k,t)\text{-lollipop}\}$-free graphs and let $G \in F \cap \mathcal{G}$.

Define
\[
m(x)=x+Cx \binom{x-1}{t}
\]
and
\[
g(x)=m(x)\Big( R(x-1,k)+\sum_{i=1}^{t-1} x\binom{x}{i} \Big).
\]
According to Property~5, we have
\[
\chi(T) \le C\,\omega \binom{\omega-1}{t}.
\]
Let $\omega$ be a positive integer.
We claim that if $\omega(G)\le \omega$, then $\chi(G)\le g(\omega)$.
We proceed by induction on $|V(G)|$.
We may assume that $\omega(G)\ge 2t-1$, because otherwise $\chi(G)\le C\le g(\omega)$.
Furthermore,
\[
\chi(G[K \cup T]) \le C\omega \binom{\omega-1}{t} + \omega.
\]
From Properties 1, 2 and 3, we have
$
|S| \le \sum_{i=1}^{t-1} \omega(G) \binom{\omega(G)}{i}, |N_{T'}(v_0)| < R(\omega-1, k)$, for each vertex $v_0\in T$.
Let
\[
\alpha = R(\omega-1,k)+ \sum_{i=1}^{t-1} \omega\binom{\omega}{i}.
\]

Then every vertex $v \in K \cup T$ has fewer than $\alpha$ neighbors in $V(G)\setminus (K \cup T)$. Let
$
c_1 : V(G \setminus (K \cup T)) \rightarrow \{1,2,\dots, g(\omega)\}
$
be a coloring obtained by the induction hypothesis. There exists a coloring
$
c_2 : K \cup T \rightarrow \{1,2,\dots,m(\omega)\}
$
of $G[K \cup T]$.

Define a coloring
$
c : V(G) \rightarrow \{1,2,\dots, g(\omega)\}
$
as follows. For $v \in V(G \setminus (K \cup T))$, let $c(v) = c_1(v)$.

Since every $v \in K \cup T$ has fewer than $\alpha$ neighbors in $V(G)\setminus(K \cup T)$,
there exists a color
$
c(v) \in \{\alpha(c_2(v)-1)+1,\dots,\alpha c_2(v)\}
$
that does not appear in $N(v)\setminus T$.

Since $c_2$ is a proper coloring of $G[K \cup T]$, it follows that $c$ is a proper coloring
of $G$ using at most $\max(\alpha m(\omega),g(\omega))$ colors. Thus, the proof is complete by Proposition~\ref{prop:Erdos1935}.
\end{proof}

\begin{theorem}\label{thm3}
Let $s,t\geq 2$ be positive integers. Suppose that $G$ is a $\{(s,t)\text{-bowtie},P_5,(s+1,t+1)\text{-dumbbell}\}$-free graph. 
Then $G$ is $(2t-2)$\text{-strongly Pollyanna}.
\end{theorem}

\begin{proof}

Let $F$ be a hereditary class of graphs and let $C$ be a positive integer
such that $\chi(G) \le C$ whenever $G \in F$ and $\omega(G) \le 2t-2$.
Let $\mathcal{G}$ be the class of $\{(s,t)\text{-bowtie},P_5,(s+1,t+1)\text{-dumbbell}\}$-free graphs and let $G \in F \cap \mathcal{G}$.
We may assume that $\omega(G) \ge 2t-1$, because otherwise $\chi(G) \le C$.\\
$
\chi(S) \le \sum_{i=1}^{t-1} C \binom{\omega(G)}{i}, 
$
Since $\omega(A_M) \le t$ (otherwise it would contradict the maximality of $K$).
By Properties 5, 6, and 7 we have
\[
\chi(G) \le C \Biggl( 2 + \sum_{i=1}^{t-1}  \binom{\omega(G)}{i} + \omega(G) \binom{\omega(G)-1}{t} \Biggr)+ \omega(G).
\]
\end{proof}

\begin{theorem}\label{thm4}
Suppose that $G$ is a $\{(2,2)\text{-bowtie}, P_5, (3,3)\text{-dumbbell}\}$-free graph with $\omega(G) \ge 3$. Then
\[
\chi(G) \le 2\omega(G) + \omega(G) \binom{\omega(G)}{2} + 2.
\]
\end{theorem}

\begin{proof}
Let $K$ be a maximum clique of size $\omega(G)$. There are $\omega(G)$ ways to choose a singleton set from $K$; let $M$ represent each such choice. 
Let $A_M$ be the set of vertices that are non-adjacent to exactly one vertex corresponding to $M$. We denote the union of these $A_M$ sets as $S$.
Furthermore, let $N$ be a set of two vertices from $K$, and let $A'_{N,v}$ be the set of vertices non-adjacent to both vertices in $N$(and adjacent to v). We denote the union of these $A'_{N,v}$ sets as $T$. Hence, $N(K) = S \cup T$. $S' \text{ is the vertices outside } K \cup S \cup T \text{ with a neighbor in } S$. $T' \text{ is the vertices outside } K \cup S \cup T \text{ with a neighbor in } T$.

Since $G$ is $P_5$-free, for any $v \in T \cup S$, we have $N^{\ge 2}(v) \setminus (T \cup S) = \emptyset$.

By Property~5,
\[
\chi(T) \le \omega(G) \binom{\omega(G)-1}{2}.
\]

By Property~6,
\[
\chi(S') \le 1.
\]

By Property~7,
\[
\chi(T') \le 1.
\]

Moreover, $A_M$ cannot contain a clique of size two, since this would contradict the maximality of $K$. The number of possible sets $A_M$ is equal to $\omega$.
Hence,
\[
\chi(S) \le \omega(G).
\]
\end{proof}

\subsection{Structure of Diamond-Free Graphs with Each Edge in at Least Two Triangles}

Let $G$ be a graph in which every edge lies in at least two triangles, and suppose that $G$ is diamond-free. For any edge $uv \in E(G)$, define
\[
K(uv) = \{u,v\} \cup \{w \mid w \sim u \text{ and } w \sim v\}.
\]
Then $K(uv)$ is a maximal clique of $G$.

Indeed, if there exist two distinct vertices $a,b \in K(uv)$ that are not adjacent, then the induced subgraph on $\{u,v,a,b\}$ forms a diamond, contradicting the assumption that $G$ is diamond-free. Therefore, $K(uv)$ is a complete graph and is maximal, and moreover, the cardinality of $K(uv)$ is at least $4$; that is, $|K(uv)| \geq 4$.

Moreover, the family of maximal cliques of $G$ forms a unique edge-disjoint partition of $E(G)$; that is, each edge of $G$ belongs to exactly one maximal clique. Otherwise, if two distinct maximal cliques shared an edge, then that edge together with two additional vertices from these cliques would induce a diamond, again contradicting the diamond-free property.
Hence, maximal cliques of $G$ may intersect in at most one vertex, but they do not share any edge.
\paragraph{Property D1 (Fan Structure at a Vertex).}
Let $A$ be a maximal clique of $G$ and let $v \in V(A)$. There may exist several other maximal cliques whose only common vertex with $A$ is $v$. This configuration is called a \emph{fan structure} at the vertex $v$, and the maximal cliques attached to $v$ are called the \emph{fan blades} (Figure~\ref{fig:ellipse-structures}a).

In this situation, no two distinct fan blades sharing the vertex \( v \) are adjacent to each other. 
Indeed, if there were an edge between two such blades, like \( ab \) where \( a \) belongs to the maximal complete subgraph \( A \) 
and \( b \) belongs to the maximal complete subgraph \( B \), then \( va \in K(vb) \), \( vb \in K(va) \). 
Hence \( a \) must be adjacent to all vertices of \( B \) and \( b \) must be adjacent to all vertices of \( A \), 
which contradicts their maximality.

\begin{theorem}\label{thm5}
Let $k \ge 1$ be an integer. Let $G$ be a graph in which every edge lies in at least two triangles.
\begin{enumerate}
\item[(a)] If $G$ is $\{\mathrm{diamond}, F(3,k)\}$-free, then $G$ is polynomially $\chi$-bounded.
\item[(b)] If $G$ is $\{\mathrm{diamond}, (4,4)\text{-dumbbell}\}$-free, then $G$ is polynomially $\chi$-bounded.
\end{enumerate}
\end{theorem}

\begin{proof}[Proof of (a)]
Let $G$ be a $\{\mathrm{diamond}, F(3,k)\}$-free graph and let $\omega \geq \omega(G) \geq 4$. We claim that there exists a function $g$ such that
$
\chi(G) \leq g(\omega).
$
We proceed by induction on $|V(G)|$.

Let $K$ be a maximum clique of $G$ of size $\omega$. Define $m(x)=x$ and
\[
g(x) = x(x-1)(k-1).
\]

By the induction hypothesis, the induced subgraph $G \setminus K$ admits a proper coloring
$
C_1 : V(G)\setminus K \to \{1,2,\dots,g(\omega)\}.
$
Let
$
C_2 : K \to \{1,2,\dots,m(\omega)\}
$
be a proper coloring of the clique $K$.

For each vertex $v \in K$, we estimate the number of neighbors of $v$ outside $K$. By Property~D1 and the fact that $G$ is $F(3,k)$-free, the number of fan blades attached to $v$ outside $K$ is strictly less than $(k-1)$ (Figure~\ref{fig:ellipse-structures}a). Each fan blade can contribute at most $(\omega-1)$ neighbors of $v$. Define
\[
\alpha := (\omega-1)(k-1).
\]
Then for every $v \in K$ we have
$
|N(v) \setminus K| < \alpha.
$

Using the same argument as in the proof of Theorem 2.7, for each vertex in $K$ there exists a color set of
size $\alpha + 1$ that avoids conflicts with the coloring of $V(G) \setminus K$.

This completes the proof of~(a).
\end{proof}

\begin{proof}[Proof of (b)]
In this case, the structure of $G$ implies that two distinct vertices of a maximum clique $K$ cannot each be adjacent to distinct maximal cliques outside $K$, since this would induce a (4,4)-$dumbbell$ subgraph (Figure~\ref{fig:ellipse-structures}b). 

Therefore, the only possible configuration is that maximal cliques are attached to a single vertex of $K$ in the form of fan blades.This structure preserves the chromatic number, maintaining equality with the clique number. This completes the proof of~(b).
\end{proof}

\begin{figure}[H]
\centering

\begin{minipage}{0.48\textwidth}
\centering
\scalebox{0.9}{
\begin{tikzpicture}[
    ellipse_style/.style={draw=green!80!black, thick, ellipse, minimum width=4cm, minimum height=1.2cm},
    ellipse_vertical/.style={draw=green!80!black, thick, ellipse, minimum width=1cm, minimum height=3cm},
    ellipse_diagonal/.style={draw=green!80!black, thick, ellipse, minimum width=3cm, minimum height=1cm, rotate=45},
    ellipse_diagonal2/.style={draw=black, thick, ellipse, minimum width=3cm, minimum height=1cm, rotate=-45},
    vertex/.style={circle, fill=black, inner sep=1.5pt},
    font=\small
]

\node[ellipse_style] (A) at (0, 0) {};
\node[vertex, label=below:$v$] (v) at (-1.5, 0) {};
\node[ellipse_vertical] (B) at (-1.5, -1.2) {};
\node[ellipse_diagonal] (C) at (-2.2, -0.8) {};

\end{tikzpicture}
}
\caption*{(a)}
\end{minipage}%
\hspace*{\fill}
\begin{minipage}{0.48\textwidth}
\centering
\scalebox{0.9}{
\begin{tikzpicture}[
    ellipse_style/.style={draw=green!80!black, thick, ellipse, minimum width=4cm, minimum height=1.2cm},
    ellipse_vertical/.style={draw=green!80!black, thick, ellipse, minimum width=1cm, minimum height=3cm},
    vertex/.style={circle, fill=black, inner sep=1.5pt},
    font=\small
]

\node[ellipse_style] (A) at (0, 0) {};
\node[vertex, label=below:$v_1$] (v1) at (-1.5, 0) {};
\node[vertex, label=below:$v_2$] (v2) at (1.5, 0) {};
\node[ellipse_vertical] (B) at (-1.5, -1.2) {};
\node[ellipse_vertical] (C) at (1.5, -1.2) {};
\draw[green!80!black, thick] (v1) -- (v2);

\end{tikzpicture}
}
\caption*{(b)}
\end{minipage}

\caption{}
\label{fig:ellipse-structures}
\end{figure}

\end{document}